\documentclass[12pt,reqno]{amsart}

\usepackage{amsmath, amssymb, amsthm, booktabs, mathtools}
\usepackage{graphicx,fullpage}
\usepackage{hyperref}
\usepackage[msc-links, initials]{amsrefs}

\newtheorem{theorem}{Theorem}
\newtheorem{proposition}{Proposition}
\newtheorem{lemma}{Lemma}
\theoremstyle{remark}
\newtheorem{remark}{Remark}
\theoremstyle{definition}
\newtheorem*{notation}{Notation}

\newcommand{\chr}{\mathrm{char}}

\hypersetup{
    colorlinks=true,
    linkcolor=blue,
    filecolor=magenta,      
    citecolor=blue,
    urlcolor=cyan,
    }

\title[Powerfree Polynomials in Short Intervals]{Short Interval Results For Powerfree Polynomials Over Finite Fields}
\author[A. Kumchev]{Angel Kumchev}
\address{Department of Mathematics\\ Towson University\\ Towson, MD 21252\\ U.S.A.}
\email{akumchev@towson.edu}
\author[N. McNew]{Nathan McNew}
\address{Department of Mathematics\\ Towson University\\ Towson, MD 21252\\ U.S.A.}
\email{nmcnew@towson.edu}
\author[A. Park]{Ariana Park}
\address{School of Mathematics\\ University of Minnesota Twin Cities\\ Minneapolis, MN 55455\\ U.S.A.}
\email{park2968@umn.edu}

\begin{document}

\maketitle

\begin{abstract}
Let $k \geq 2$ be an integer and $\mathbb F_q$ be a finite field with $q$ elements. We prove several results on the distribution in short intervals of polynomials in $\mathbb F_q[x]$ that are not divisible by the $k$th power of any non-constant polynomial. Our main result generalizes a recent theorem by Carmon and Entin \cite{CaEn21} on the distribution of squarefree polynomials to all $k \ge 2$. We also develop polynomial versions of the classical techniques used to study gaps between $k$-free integers in $\mathbb Z$. We apply these techniques to obtain analogues in $\mathbb F_q[x]$ of some classical theorems on the distribution of $k$-free integers. The latter results complement the main theorem in the case when the degrees of the polynomials are of moderate size.    
\end{abstract}

\section{Introduction}

Recall that if $k \ge 2$ is a fixed integer, an integer $n$ is called $k$-free if $n$ is not divisible by the $k$th power of any prime. This is a generalization of the classical concept of a squarefree integer, which occurs in the special case when $k=2$. Much work has been done studying the distribution of $k$-free integers in short intervals, especially in the squarefree case: see \cites{fogel41, erdos51, Erdo66,  Fila88, Fila90, FGT15, FT89, FT92, FT96,  gk88, Gran98, Halb83, HR51, KMMPSZ1, Rank55, richert54, Roth51, Schm64, tr89, Tr95}. In particular, Filaseta and Trifonov~\cite{FT92} proved that there exists a constant $c > 0$ such that the interval $(x,x+cx^{1/5}\ln x]$ contains a squarefree integer for all sufficiently large $x$. Trifonov~\cite{Tr95} further generalized this result to $k$-free integers for all $k \ge 2$. He showed that for some constant $c = c(k) > 0$, the interval $(x,x+cx^{1/(2k+1)}\ln x]$ contains a $k$-free integer when $x$ is sufficiently large. To the best of our knowledge, these are the sharpest unconditional upper bounds on the maximum gap between consecutive $k$-free numbers. Conditionally on the $abc$-conjecture, Granville~\cite{Gran98} has shown that for any fixed $\varepsilon > 0$, the interval $(x, x + x^{\varepsilon}]$ contains squarefree integers for sufficiently large $x$. 

There are many parallels between the arithmetic of $\mathbb Z$ and that of $\mathbb F_q[x]$, the ring of polynomials in $x$ over a finite field $\mathbb F_q$ with $q$ elements (see \cites{LN97, Rosen02} for background on such research). In particular, a polynomial in $\mathbb F_q[x]$ is called {\em $k$-free} if it has no irreducible factors of multiplicity $k$ or higher; when $k=2$, we call such a polynomial {\em squarefree}. One may expect to find ample existing research on analogues for polynomials from $\mathbb F_q[x]$ of the aforementioned research on the gap problem for $k$-free integers, but that does not appear to be the case. Indeed, a search of the literature on the distribution in short intervals of $k$-free polynomials over a finite field yields very limited results, almost entirely focused on the squarefree case. 

Let $q = p^f$, with $p$ prime and $f \in \mathbb N$, be the cardinality of a finite field $\mathbb F_q$. Henceforth, we restrict $q$ to integers of this form. We let $\mathcal M_q$ denote the set of monic polynomials in $\mathbb F_q[x]$ and write $\mathcal M_q(d)$ for the subset of monic polynomials of degree $d$. When $F\in\mathcal M_q$ and $h < \deg F$, an interval in $\mathbb F_q[x]$ of length $h$ 
centered at $F$ is the set
\[\mathcal I_q (F,h) = \left\{ Q\in\mathbb F_q[x] : \deg(F-Q) \leq h \right\}. \]
In this paper, we study $k$-free polynomials in short intervals of this kind. To draw an analogy with short intervals $(x,x+h]$ in $\mathbb Z$, we observe that when $F \in \mathcal M_q(n)$, the ``size'' of the polynomials in $\mathcal I_q (F,h)$ is $q^n$, whereas the number of polynomials in the interval is $q^{h+1}$. In particular, the interval is ``short'' whenever $0 < h \le n-2$. Thus, a proper analogue of a short interval $(x,x+h]$, where $x \to \infty$ and $h = O(x^{\theta})$, $0 < \theta < 1$, is an interval $\mathcal I_q (F,h)$, where $q^n \to \infty$ and $h \le \theta n$. 

Note that the condition $q^n \to \infty$ above can occur in different ways. For example, one may fix $n = \deg(F)$ and let $q \to \infty$. In this regime, the question was studied by Keating and Rudnick \cite{KR16}. Drawing on earlier work by Rudnick \cite{Rudn14} on the density of squarefree polynomials over $\mathbb F_q$, they showed that for any integers $h,n$ with $0 < h \le n-2$, one can take $q$ sufficiently large so that there exists a squarefree polynomial in every interval $\mathcal{I}_q(F,h)$, with $F \in \mathcal M_q(n)$. The theorem of Keating and Rudnick does not quantify how fast $q$ must grow in terms of $n$, but an examination of their proof suggests that it can be made effective to show that such a conclusion holds as long as $q > c(n+h)$ for some constant $c$.

In this paper, we focus on the case when $q$ is fixed and $n \to \infty$. The behavior of powerfree polynomials in this regime turns out to be quite different, and the analogy with $\mathbb Z$ is more direct. For example, in the case of gaps between squarefree integers, Erd\H{o}s \cite{erdos51} proved long ago that the maximum gap is unbounded: there are arbitrarily large $x$ such that the interval $(x,x+h]$ contains no squarefree integers when 
\[ h \leq \frac {c\ln x}{\ln\ln x} \]
for any constant $c$ such that $2c < \zeta(2)$. In \S \ref{sec:erdosbound}, we establish a version of Erd\H{o}s' result for polynomials over $\mathbb F_q$. If $\zeta_q(s) = (1 - q^{1-s})^{-1}$ denotes the zeta-function of the ring $\mathbb F_q[x]$ (see \cite{Rosen02}*{Ch. 2}), our result can be stated as follows.

\begin{theorem}\label{prop:crt}
Let $k \geq 2$ and $q \geq 2$ be fixed integers, and suppose that $c$ is any constant with $kc < \zeta_q(k)$. If $n$ is sufficiently large, there exist monic polynomials $F$ of degree at most $n$ such that the interval $\mathcal{I}_q(F,h)$ contains no $k$-free polynomials for any length $h$ such that
\begin{equation}
q^{h+1} \leq \frac {cn}{\log_q n}. \label{eq:hbound}
\end{equation}
\end{theorem}

In the squarefree case $k=2$, this is a direct analogue of Erd\H{o}s' result, as stated by Erd\H{o}s in \cite{Erdo66}. To the best of our knowledge, for $k > 2$, the corresponding result for integers has never been formally stated, though it has been known to researchers in the field and can be extracted from the remarks in \cite{Erdo66}. 

We include Theorem \ref{prop:crt} and its proof here, since it transpires that in the study of $k$-free polynomials over $\mathbb F_q$, the upper bounds on the least $h$ (as $n \to \infty$) for which $\mathcal I_q(F,h)$ must contain a $k$-free polynomial come much closer to the lower barrier imposed by Theorem \ref{prop:crt}. Recently, Carmon and Entin \cite{CaEn21} have shown that when
\begin{equation}\label{eq:carmon-cond}
q^{h+1}  > \left( \frac {g(n) n}{\log_q n} \right)^p, 
\end{equation} 
where $g(n) \to \infty$ as $n \to \infty$, one can obtain an asymptotic formula for the number of squarefree polynomials in the interval $\mathcal I_q(F,h)$. They derive this result as a special case of a theorem on the density of squarefree values of bivariate polynomials over $\mathbb F_q$. In particular, their proof is considerably more elaborate than is necessary for the application to the gap problem considered here. In the special case of interest, we developed a much simplified variant of their method, which we present in \S\ref{sec:carmon-basic}; it yields a rather quick proof that when $n \to \infty$ and \eqref{eq:carmon-cond} holds with $g(n)=1$, the interval $\mathcal I_q(F,h)$ contains (many) squarefree polynomials. Then, in \S\ref{sec:carmon}, we extend the method to $k$-free polynomials, for any $k \geq 2$, and establish the following result. 

\begin{theorem}\label{thm:main}
Let $k \ge 2$ and $q \ge 2$ be fixed integers, and suppose that $\chr(\mathbb F_q) = p$. Let $k = dp^a + \dots + d_1p + d_0$, $0 \le d = d_a, \dots, d_1, d_0 < p$, $d \ne 0$, be the base-$p$ representation of $k$. If $n$ is sufficiently large and $F \in \mathcal{M}_q(n)$, the interval $\mathcal{I}_q(F,h)$ contains a $k$-free polynomial whenever 
\begin{equation}\label{eq:main-cond}
q^{h+1}  > \left( \frac {n}{\log_q n} \right)^{1/\theta}, 
\end{equation} 
where $\theta = 1 - ({p - d +1}){p^{-a-1}}$. \end{theorem}

Note that when $k=2$, we have $\theta=p^{-1}$, and inequality \eqref{eq:main-cond} becomes $\eqref{eq:carmon-cond}$ with $g(n)=1$.  In general, $\theta$ is a non-decreasing function of $k$ such that $\theta=(k-1)p^{-1}$ when $2 \le k \le p$, and 
\[  1 - \frac {(d+1)(p-d+1)}{pk} < \theta \le 1 - \frac 1k \] 
when $k > p$. In particular, as $k$ increases, the gap between the barrier imposed by \eqref{eq:hbound} and the hypothesis \eqref{eq:main-cond} of Theorem~\ref{thm:main} shrinks, and our result gets closer to being best possible. 

The method of proof of Theorem \ref{thm:main} can be adjusted to yield variants that are superior in different ways. As stated, the theorem is close to the best result one can obtain from the basic version of our method. This lets us avoid some technical details. However, as we note at the end of~\S\ref{sec:carmon}, if one is interested in an asymptotic for the number of $k$-free polynomials in $\mathcal I_q(F, h)$, similar to that in the original work of Carmon and Entin \cite{CaEn21}, one may obtain such an asymptotic for $n \to \infty$ at the cost of strengthening condition \eqref{eq:main-cond} to 
\begin{equation}\label{eq:main-cond1}
q^{h+1}  > \left( \frac {g(n)n}{\log_q n} \right)^{1/\theta} 
\end{equation} 
with $g(n) \to \infty$. One can also relax hypothesis \eqref{eq:main-cond} to \eqref{eq:main-cond1} with $g(n) = c$, where $c$ is any constant satisfying
$c > \theta \zeta_q(k)p^{-a}$. As $\theta \zeta_q(k)p^{-a} < 1$, this is a slight improvement on Theorem~\ref{thm:main}.

A notable feature of the modern results on gaps between $k$-free integers is that they can be made fully explicit. For example, in recent joint work with McCormick, Scherr, and Ziehr \cite{KMMPSZ1}, the authors proved an explicit version of the theorem of Filaseta and Trifonov~\cite{FT92}: the main result of \cite{KMMPSZ1} establishes that the interval $(x, x + 11x^{1/5}\ln x]$ contains a squarefree integer for any $x \ge 2$. The next theorem provides a model for such results for polynomials over $\mathbb F_q$. Note that---in contrast to Theorems \ref{prop:crt}, \ref{thm:main}, and \ref{thm:k} and similar to the main result of~\cite{KMMPSZ1}---this theorem makes the restriction on the size of the degree $n$ explicit. 

\begin{theorem}\label{thm:n3}
Let $k \ge 2$ and $q \ge 3$ be fixed integers. If $n \ge k+1$ and $F \in \mathcal M_q(n)$, the interval $\mathcal I_q(F,h)$ contains a $k$-free polynomial for all $h \geq n/(k+1)$. 
\end{theorem}

When $k=2$, this theorem corresponds to the classical result that the interval $(x,x+x^{1/3}]$ contains a squarefree integer for all sufficiently large $x$. A slightly stronger version of this was first proved by Davenport in 1951, but not published at the time; its elementary and rather elegant proof can be found in Halberstam's survey \cite{Halb83}. While all the estimates in the proof of Theorem \ref{thm:main} can be made fully explicit, thus allowing us to quantify the hypothesis that ``$n$ is sufficiently large,'' the method is not suited to yield non-trivial results when $n$ and $h$ are as small as they can be in Theorem \ref{thm:n3}.  See Table \ref{table1} for a comparison of the values of $h$, $n$ and $q$ for which the results developed in this paper are applicable.   We prove this result using a variant for polynomials over finite fields of a differencing technique introduced by Halberstam and Roth \cites{HR51, Roth51} and later developed by Filaseta and Trifonov \cites{FGT15, FT89, FT92, Tr95}. The proof of Theorem \ref{thm:n3} requires only the most basic form of the differencing method. A slightly more sophisticated version of those ideas yields the following result.

\begin{theorem}\label{thm:n4}
Let $k \ge 2$ and $q \ge 7$ be fixed integers such that $ \chr(\mathbb F_q) \nmid (k+1)$. If $n \ge k+1$ and $F \in \mathcal M_q(n)$, the interval $\mathcal I_q(F,h)$ contains a $k$-free polynomial for all $h \geq n/(k+2)$. Moreover, the same conclusion holds when $k \ge 3$ and $q \ge 5$.
\end{theorem}

In the case $k=2$, this result matches a theorem due to Roth \cite{Roth51} (after some modification by Nair \cite{nair79}) that the interval $(x, x+cx^{1/4}]$ contains a squarefree integer for some absolute constant $c > 0$. When $k \ge 3$, however, Theorem \ref{thm:n4} falls short of matching  the theorem of Halberstam and Roth \cite{HR51} that, for any fixed $\varepsilon > 0$, the interval $(x, x + x^{1/(2k) + \varepsilon}]$ contains $k$-free integers when $x$ is sufficiently large. The next theorem accomplishes this.

\begin{theorem}\label{thm:k}
Let $k\geq 3$ and $q \ge 3$ be fixed integers such that $\chr(\mathbb F_q) \nmid k\binom {2k-1}{k-1}$. If $n$ is sufficiently large and $F \in \mathcal M_q(n)$, then the interval $\mathcal I_q(F,h)$ contains a $k$-free polynomial for all $h \geq n/(2k)$.  
\end{theorem}

While this theorem matches the Halberstam--Roth result in terms of the sizes of the intervals, it is much weaker than Theorem \ref{thm:main} (and, unlike Theorem \ref{thm:n4}, it says nothing about polynomials of small degrees). On the other hand, its proof adapts the method used by Halberstam and Roth in their seminal paper \cite{HR51} (as presented in \cite{FGT15}). It also demonstrates how one may develop further the ideas behind Theorems \ref{thm:n3} and \ref{thm:n4}. In the integer setting, it is more advanced versions of those ideas that yield the best results by Filaseta and Trifonov on gaps between $k$-free integers. Indeed, Filaseta and Trifonov (see \cites{FGT15, FT96}) have used those ideas to make progress in other problems, and it is conceivable that further applications may exist in the function field setting too. For these reasons, it seems that the proof of Theorem~\ref{thm:k} is of independent interest (even though the result itself is superseded by Theorem \ref{thm:main}), and so it appears as an appendix to this paper.\footnote{It is possible to generalize the improvements of Filaseta and Trifonov to the polynomial setting as well.  These methods can be used to show that the interval $\mathcal I_q(F,h)$ contains a squarefree polynomial for all $h \ge n/5  + \log_qn$ when $n$ is sufficiently large and $p>3$. This result is strictly weaker than Theorem \ref{thm:main} and the proof substantially more involved, so we will not pursue it further here.}

The remainder of the paper is organized as follows. In \S\ref{sec:prelim}, we 
present the basic setup for the proofs and gather some preliminary facts about polynomials over finite fields. We also present present the proofs of Theorem \ref{thm:main} for $k=2$ and of Theorem ~\ref{thm:n3} in the case when $k$ is not divisible by the characteristic. In \S\ref{sec:erdosbound}, we establish Theorem \ref{prop:crt}. The proof of Theorem~\ref{thm:main} in the general case appears in \S\ref{sec:carmon}. In \S\ref{sec:spacings}, we develop polynomial analogues of the basic form of the methods used by Filaseta and Trifonov in their work on the gap problem for $k$-free integers: see Propositions \ref{prop1} and \ref{prop:4d-n/3} below. We then apply those results to prove Theorems~\ref{thm:n3} and~\ref{thm:n4}. Finally, as noted earlier, the appendix contains the proof of Theorem \ref{thm:k}, including our version of the Halberstam--Roth method (see Proposition \ref{prop4}).  

\begin{notation}  
Throughout the paper, the finite field $\mathbb F_q$ is considered fixed, and we use $p$ to denote its characteristic (so that, $q = p^f$ for some $f \in \mathbb N$). Beside the sets of monic polynomials $\mathcal M_q$ and $\mathcal M_q(d)$, we use $\mathcal P_q$ to denote the set of monic irreducible polynomials and $\mathcal P_q(d)$ the set of monic irreducible polynomials of degree $d$. We write $\pi_q(d)=|\mathcal P_q(d)|$ for the number of monic irreducible polynomials of degree $d$ in $\mathbb F_q[x]$; in general, $|\mathcal A|$ denotes the cardinality of a finite set $\mathcal A$. 
\end{notation}

\section{Preliminaries}
\label{sec:prelim}

Fix an integer $k \ge 2$. Our strategy to prove the existence of $k$-free polynomials in an interval $\mathcal I_q(F, h)$ will be to bound from above the number $\mathcal N_{q}(F,h)$ of polynomials in $\mathcal I_q(F,h)$ that are not $k$-free and to show that 
\begin{equation}\label{eqn2.0}
\mathcal N_{q}(F,h) <  |\mathcal{I}_q(F,h)| = q^{h+1}.     
\end{equation} 
Since every polynomial that is not $k$-free is divisible by the $k$th power of some monic irreducible polynomial (and the $k$th power of a polynomial of degree greater than $n/k$ cannot divide any polynomial in $\mathcal I_q(F,h)$), we find that
\begin{equation}\label{eqn:bound}
\begin{split}
    \mathcal N_q(F,h) &\leq \sum_{P \in \mathcal P_q} |\{Q\in \mathcal I_q(F,h) : P^k\mid Q \}| \\
    &= \sum_{d \le n/k} \sum_{P \in \mathcal P_q(d)} |\{Q\in \mathcal I_q(F,h) : P^k\mid Q \}|.  
\end{split}
\end{equation}
 
It will be useful to recall how many polynomials in $\mathcal{I}_q(F,h)$ are divisible by a fixed polynomial $G$.

\begin{lemma}\label{lem2.1}
Suppose $G \in \mathcal{M}_q(d)$. Then either $\mathcal I_q(F,h)$ contains no multiple of $G$, or 
\begin{equation}
|\{Q\in \mathcal I_q(F,h) : G \mid Q \}| 
= \begin{cases}
q^{h-d+1} & \text{if } d \leq h,\\
1 & \text{if } d > h.
\end{cases} \label{eq:gmultcount} 
\end{equation}
\end{lemma}

\begin{proof}
Suppose that $GA \in \mathcal I_q(F,h)$ for some polynomial $A \in \mathcal M_q$. When $d > h$, the interval can contain no other multiples of $G$; and when $d \le h$, we need to count the polynomials $GB$, with $B \in \mathcal{I}_q(A,h-d)$.
\end{proof}

The above lemma suffices to estimate the contribution to the right side of \eqref{eqn:bound} from irreducible polynomials $P$ of degrees $d \le \ell$, when $\ell$ is not much larger than $h$. We have
\begin{equation}\label{eqn2.3}
\sum_{d \le \ell} \sum_{P \in \mathcal P_q(d)} |\{Q\in \mathcal I_q(F,h) : P^k\mid Q \}| 
\le \sum_{d \leq h/k} \pi_q(d)q^{h-kd+1} \ + \sum_{h/k < d \le \ell} \pi_q(d) =: \Sigma_1 + \Sigma_2. 
\end{equation}
To bound $\Sigma_1$, $\Sigma_2$, and other similar sums below, we will use some well-known bounds for $\pi_q(d)$, which we state in the next lemma. The first claim of this lemma can be found in \cite{LN97}*{Corollary~3.21}, and the second claim is an immediate consequence of the first.

\begin{lemma}\label{lemma-prime-count}
For any natural number $n$, one has
\[ \sum_{d \mid n} d\pi_q(d) = q^n \quad \text{and} \quad \pi_q(n)\leq\frac{q^n}{n}. \]
\end{lemma}

Suppose that $k \le h$. Using this lemma, we find that
\begin{align}\label{eq:sig1bd}
\Sigma_1 &\le \sum_{d \leq h/k} \frac {q^{h+1} }{dq^{(k-1)d}} < q^{h+1} \ln\left( \frac {1}{1-q^{1-k}} \right) = q^{h+1}\ln \zeta_q(k),
\end{align}
and (assuming that $\ell \ge h$)
\begin{equation}\label{eq:sig2bd}
\Sigma_2 \le \sum_{h/k < d \leq \ell} \frac {q^d}d  < \frac {q^\ell}{\ell} + \frac {k}h \sum_{j = 0}^{\infty} q^{\ell - 1 -j} 
\leq \frac {q^\ell}{h} + \frac {kq^{\ell}}{(q-1)h} = \frac {(q+k_h-1)q^{\ell}}{(q-1)h},
\end{equation}
where $k_h \coloneqq \min(k,h)$. When $k > h$, the sum $\Sigma_1$ is empty, while $\Sigma_2$ satisfies the same bound, after a small adjustment to its proof:
\[ \Sigma_2 \le \sum_{d \leq \ell} \frac {q^d}d  < \frac {q^\ell}{h} + \frac {q^{\ell}}{(q-1)} = \frac {(q+k_h-1)q^{\ell}}{(q-1)h}. \]

\subsection{The classical approach}
\label{sec:2.1}

Returning to the contribution to the right side of \eqref{eqn:bound} from degrees $d>h$, we may apply Lemma \ref{lem2.1} to show that when $h < d \leq n/k$, we have
\begin{equation}
    \sum_{P \in \mathcal P_q(d)} |\{Q\in \mathcal I_q(F,h) : P^k\mid Q \}| \leq |\mathcal S_q(d)|, \label{useTd}
\end{equation}
where
\begin{equation}\label{def:Td}
\mathcal S_q(d) = \big\{ G \in  \mathcal{M}_q(d) : G^kA \in\mathcal I_q(F, h) \text{ for some } A \in\mathcal M_q \big\}. 
\end{equation} 

The shift of focus from the polynomials in $\mathcal I_q(F,h)$ to their $k$th-power divisors that occurs in inequality \eqref{useTd} is an  $\mathbb F_q[x]$-variant of the basic idea at the core of the proofs of most bounds on gaps between $k$-free integers mentioned in the introduction. In later sections, we prove several results about the ``spacing'' between polynomials divisible by $k$th powers as measured by the degrees of the differences between their $k$th-power factors. Such spacing results lead to upper bounds on $|\mathcal S_q(d)|$ through the following lemma.

\begin{lemma}\label{lem2.3}
Let $\mathcal S \subseteq \mathcal M_q(d)$, and suppose that $\kappa,\delta \in \mathbb R^+$, $\delta \leq d$, have the following property: for any fixed polynomial $G \in \mathcal S$, there exist at most $\kappa$ polynomials $H \in \mathcal S$ such that $\deg(G-H) < \delta$. Then     
\[ |\mathcal S| \leq \kappa q^{d-\delta}. \]
\end{lemma}

\begin{proof}
Choose $k \in \mathbb N$ so that $k-1 < \delta \le k$. The intervals $\mathcal I_q( x^{k}Y(x), \ k-1)$, with $Y \in \mathcal M_q(d-k)$, form a partition of $\mathcal M_q(d)$. Let $\mathcal I$ be one such interval, and fix a polynomial $G \in \mathcal S \cap \mathcal I$. Since any two elements $G, H$ of $\mathcal S \cap \mathcal I$ must satisfy $\deg(G - H) < \delta$, by hypothesis, there are at most $\kappa$ possible polynomials $H \in \mathcal S \cap \mathcal I$, including $G$ itself. Thus,
\[ |\mathcal S \cap \mathcal I| \le \kappa. \]
Summing this estimate over all $q^{d-k} \le q^{d-\delta}$ intervals $\mathcal I$ of the above form, we get the desired bound.    
\end{proof}

For example, in Section \ref{sec:spacings}, we will show that when $p \nmid k$---and so Proposition \ref{prop1} holds with $r= 1$, any two distinct polynomials $G, H \in \mathcal S_q(d)$ satisfy $\deg(G-H) \ge (k+1)d-n$. 
Thus, when $d > n/(k+1)$, we may apply the above lemma with $\kappa = 1$ and $\delta = (k+1)d-n$ to obtain
\begin{equation}
|\mathcal S_q(d)|\leq q^{n-kd}. \label{bd1}    
\end{equation}
This bound suffices to give a quick proof of Theorem \ref{thm:n3} in the case when $k$ is not divisible by the characteristic.  The proof in the case $p\mid k$ will appear in Section \ref{sec:spacings}.

\begin{proof}[Proof of Theorem \ref{thm:n3}: The case $p \nmid k$]
When $h \ge n/(k+1)$, the condition $d > h$ implies $d \ge (n+1)/(k+1)$. So, we may apply \eqref{bd1} to all $d$ in the range $h < d \le n/k$ to get
\begin{align*}
\sum_{h < d \leq n/k} |\mathcal S_q(d)| &\leq \sum_{h < d \leq n/k} q^{n-kd} < q^{n-k(n+1)/(k+1)} \sum_{j \ge 0} q^{-kj} = \frac {q^{(n+k^2)/(k+1)}}{q^k-1} \leq \frac {q^{h + 1 + 1/(k+1)}}{q^2-q^{2-k}}.
\end{align*} 
Combining this bound with \eqref{eqn:bound} and \eqref{eqn2.3}--\eqref{useTd} with $\ell = h $, we find that
\begin{equation}\label{eq:n/3-bnd}
\mathcal N_q(F,h) \leq q^{h+1} \left( \ln \zeta_q(k) + \frac {q+k_h-1}{q(q-1)h} + \frac {q^{1/(k+1)}}{q^2-q^{2-k}}\right).    
\end{equation}  

When $k=2$, this establishes \eqref{eqn2.0} when $q \ge 5$ and $h \ge 1$ or when $q = 3$ and $h \ge 2$. Similarly, when $k \ge 3$, this inequality proves the theorem when $q \ge 3$. When $k=2$, $q=3$, and $h=1$, we are in the case $k > h$, so by our earlier observation, $\Sigma_1$ is empty and the logarithmic term on the right side of \eqref{eq:n/3-bnd} is superfluous. The stronger version of \eqref{eq:n/3-bnd} that results from its omission establishes the theorem in this last remaining case. 
\end{proof}

\subsection{The Carmon--Entin approach}
\label{sec:carmon-basic}

We now present a simplified version of the method of Carmon and Entin \cite{CaEn21}, which gives a quick proof of  Theorem \ref{thm:main} in the squarefree case $k=2$ for $q > 2$. (With small adjustments, the method can be applied to the case $q = 2$ as well, but we defer that discussion to the general proof in \S\ref{sec:carmon}.)

The method relies on two main observations. First, we note that when $G = P^2A$ for some polynomials $P$ and $A$, we have also $P \mid G'$, since $G' = 2PP'A + P^2A'$. This simple observation is central also to the proofs in \cite{CaEn21} of the more general theorems there.

Our second observation, which replaces a more elaborate construction in \cite{CaEn21}, is that 
in characteristic $p$, the coefficients of the monomials $x^{p-1}, x^{2p-1}, \dots$ in $G'$ vanish, and so $G' \in \mathcal{D}^{p-1}_q$,
where 
\[ \mathcal{D}_q^j \coloneqq \left\{ a_mx^{m}+ \dots + a_0 \in \mathbb F_q[x] \mid a_i = 0 \text{ if } i \equiv j \!\!\! \pmod {p}\right\}. \] 

Let  
\[ \Sigma_3 = |\{G \in \mathcal{I}_q(F,h):  P^2 \mid G \text{ for some } P \in \mathcal{P}_q, \, \deg(P) > h \} |. \]
By the above observations, any polynomial $G$ counted by $\Sigma_3$ has derivative $G' $ lying in $\mathcal I_q(F', h-1) \cap \mathcal D_q^{p-1}$, and $G'$ shares an irreducible factor $P$ with $G$, where $\deg P > h$. Note that
\[ |\mathcal{I}_q(F',h-1) \cap \mathcal D_q^{p-1} | = q^{h-\lfloor h/p\rfloor}, \]
which is significantly smaller than the total number of polynomials in $\mathcal I_q(F,h)$. Next, we show that the number of polynomials counted by $\Sigma_3$ is not much larger than the number of their derivatives.

Fix $H \in \mathcal{I}_q(F',h-1) \cap \mathcal D_q^{p-1}$, and let $P$ be an irreducible divisor of $H$ of degree at least $h+1$. Note that $P$ can divide at most one such polynomial $H$, so $P^2$ divides a polynomial in $\mathcal{I}_q(F,h)$ if and only if $H$ has an antiderivative in this interval which is divisible by $P$. Since, $H \in \mathcal D_q^{p-1}$, each monomial $a_ix^i$ of $H$ has an ``obvious'' antiderivative $a_i(i+1)^{-1}x^{i+1}$; let $H_0$ be the resulting antiderivative of $H$. The general antiderivative of $H$ is $H_0+C$ for any polynomial $C \in \mathbb F_q[x^p]$.  So, $H$ has an antiderivative divisible by $P$ in the interval $\mathcal I_q(F,h)$ if and only if there is a polynomial $C \in \mathbb F_q[x^p]$ that lies in the congruence class $C \equiv -H_0 \pmod{P}$ (in $\mathbb F_q[x]$) and is such that $H_0+C \in \mathcal{I}_q(F,h)$. Since $\deg P > h$, if such polynomials exist, at most one can lie in the interval $\mathcal I_q(F,h)$. 

Thus, whenever $H \neq 0$, each irreducible factor of $H$ of degree at least $h+1$ corresponds to at most one polynomial $G$ counted by $\Sigma_3$. Since $H$ has degree at most $n-1$, it has $ \leq (n-1)/(h+1)<n/(h+1)$ such irreducible factors.  On the other hand, if $H$ is identically zero (which can happen only when $\mathcal{I}_q(F,h)$ contains a $p$th power), then $H$ has exactly $q^{\lfloor h/p\rfloor + 1}$ antiderivatives in $\mathcal{I}_q(F,h)$. We conclude that 
\begin{align*}
    \Sigma_3 \leq \frac{n}{h+1}q^{h-\lfloor h/p\rfloor} + q^{\lfloor h/p\rfloor + 1}.
\end{align*}

Combining the last bound with the estimates for $\Sigma_1$ and $\Sigma_2$ in \eqref{eq:sig1bd} and \eqref{eq:sig2bd} with $k=2$ and $\ell = h$, we find that 
\begin{equation} \label{eq:sqfreeintbd}
    \mathcal{N}_q(F,h) \leq q^{h+1}\left( \ln\left( \frac q{q-1} \right) + \frac {q+1}{(q-1)qh} + \frac {nq^{-(h+1)/p}}{h+1} + q^{h(1/p-1)} \right).
\end{equation} 
When $h+1 \geq p(\log_q n - \log_q \log_q n)$ (this is equivalent to \eqref{eq:main-cond} with $\theta = p^{-1}$), we have
\[ (h+1)q^{(h+1)/p} \ge pn \left( 1 - \frac {\log_q \log_q n}{\log_q n} \right), \]
and our bound on $\mathcal N_q(F,h)$ simplifies to
\[ \mathcal{N}_q(F,h) \leq q^{h+1} \left( \ln\left( \frac q{q-1} \right) + \frac 1{p} + O\left(\frac {\log_q \log_q n}{\log_q n} \right)\right). \]
When $n$ is large and $q > 2$, this proves \eqref{eqn2.0} and establishes Theorem \ref{thm:main}. 

\section{Intervals without $k$-free polynomials}
\label{sec:erdosbound}

In this section, we establish the polynomial analog of Erd\H{o}s' result on large gaps between squarefree integers stated in Theorem \ref{prop:crt}. In its proof, we make use of the following lemma, which can be found in \cite{GKMZ}*{Theorem 4.1}.

\begin{lemma} \label{lem:kthirred}
Let $P_1, P_2, \dots, P_j, \dots$ be any ordering of the irreducible monic polynomials in $\mathbb F_q[x]$ such that $\deg P_j \le \deg P_{j+1}$. Then, as $j \to \infty$,
\[ \deg{P_j} \leq \log_q{j} + \log_q{\log_q{j}} +\log_q{(q-1)}+o(1). \]
\end{lemma}

We also count precisely the number of polynomials in an interval covered by congruences modulo powers of irreducible polynomials of small degrees.

\begin{lemma} \label{lem:exactsieve}
    Let $k\geq 2$, $\ell \leq \log_q(h/k) -1$, and fix a congruence class $Q_j \pmod{ P_j^k}$ for every irreducible polynomial $P_j$ with $\deg P_j \leq \ell$.  Then the number of polynomials in any interval $\mathcal{I}_q(F,h)$ satisfying at least one of these congruences is exactly  
    \[ q^{h+1} \bigg( 1 - \prod_{d=1}^\ell\prod_{d \in \mathcal P_q(d)} \bigg( 1 - \frac 1{q^{kd}} \bigg) \bigg) =q^{h+1}\left(1-\frac{1}{\zeta_q(k)}+O\left(\frac{1}{\ell q^\ell}\right)\right). \]
\end{lemma}

\begin{proof}
Define 
$ M = \prod\limits_{d=1}^\ell \prod\limits_{P \in \mathcal P_q(d)} P^k$.
We find that
\[ \deg M = \sum_{d=1}^\ell kd\pi_q(d) \le \sum_{d=1}^\ell kq^d  < \frac{kq^{\ell+1}}{q-1} \leq \frac{h}{q-1} \le h. \] 
Thus, we can apply the inclusion-exclusion principle and Lemma \ref{lem2.1} to get an exact count of the polynomials in $\mathcal I_q(F, h)$ covered by the congruence classes $Q_j \bmod P_j^k$. In particular, since $\deg M < h$ such an interval will always contain exactly
\[ q^{h+1} \bigg( 1 - \prod_{d=1}^\ell\prod_{d \in \mathcal P_q(d)} \bigg( 1 - \frac 1{q^{kd}} \bigg) \bigg) = q^{h+1}\left(1-\prod_{d=1}^\ell \bigg( 1 - \frac 1{q^{kd}} \bigg)^{\pi_q(d)}\right)\]
polynomials satisfying at least one of the congruences.  

Using that $\prod_{P \in \mathcal P_q} \big( 1 - q^{-k\deg P} \big) = \zeta_q(k)^{-1} = 1 - q^{-k+1}$, and the estimate
\[ \sum_{d > \ell} \ln\left( 1 - \frac 1{q^{kd}} \right)^{-\pi_q(d)} \leq \sum_{d > \ell} \frac {q^d}{d} \ln\left( 1 - \frac 1{q^{kd}} \right)^{-1} \ll \frac 1{\ell q^\ell}. \]
we find that
\begin{align*}
\prod_{d=1}^\ell \bigg( 1 - \frac 1{q^{kd}} \bigg)^{\pi_q(d)} &= (1 - q^{1-k}) \prod_{d > \ell} \left( 1 - \frac 1{q^{kd}} \right)^{-\pi_q(d)} \\
&= \zeta_q(k)^{-1}+O\left(\frac{1}{\ell q^\ell} \right)
\end{align*}
and the result follows.
\end{proof}

\begin{proof}[Proof of Theorem \ref{prop:crt}]
Let $h$ be large. We will use the Chinese Remainder Theorem to construct a polynomial $F$ of degree at most $n$ such that no polynomial in $\mathcal I_q(F,h)$  is squarefree when $h$ satisfies \eqref{eq:hbound}. 

The construction is based on a simple idea. Let $P_1, P_2, \dots$ be an ordering of $\mathcal P_q$ such that $\deg P_j \le \deg P_{j+1}$. If $Q_1, Q_2, \dots, Q_m$ are any polynomials such that the congruence classes $Q_j \bmod P_j^k$, $j \le m$, cover the interval $\mathcal I_q(0,h)$, then the interval $\mathcal I_q(F,h)$ contains no $k$-free polynomial whenever $F$ satisfies the congruences
\[ F \equiv -Q_j \pmod {P_j^k} \qquad (1 \le j \le m). \]
Since the Chinese Remainder Theorem determines such a polynomial $F$ modulo $P_1^k \cdots P_m^k$, we can find a nontrivial solution of these congruences of degree $\le \deg(P_1^k\cdots P_m^k)$. We can use Lemma \ref{lem:kthirred} to bound the degree of such a polynomial $F$. We have
\begin{align*}
\deg F &\leq k\sum_{j=1}^m \deg P_j 
\leq k\sum_{j=1}^m (\log_q j + \log_q\log_q j +\log_q(q-1)+ o(1)) \\
&= \frac k{\ln q} \ln(m!) + km\log_q\log_q m + km\log_q(q-1)+ o(m)\\
&= km\left(\log_q m + \log_q\log_q m + \log_q\left(\frac {q-1}e \right) + o(1)\right) =: \delta(m), 
\end{align*}
where the last step uses Stirling's formula. Thus, the proposition will follow, if we show that condition \eqref{eq:hbound} allows us to find an integer $m$ with $\delta(m) \le n$ and $m$ congruence classes $Q_j \bmod P_j^k$ that cover $\mathcal I_q(0,h)$.

The simplest way to find such a congruence cover is to use a separate congruence class for every polynomial in $\mathcal I_q(0,h)$. Then $m = q^{h+1}$. This already suffices to establish the theorem when the constant $c$ in \eqref{eq:hbound} satisfies $c < (kq)^{-1}$. It is clear, however, that this simple argument is somewhat wasteful. Next, we use Lemma \ref{lem:exactsieve} to cover multiple polynomials by congruences modulo small irreducible polynomials. 

Define $\ell \coloneqq \lfloor \log_q (h/k) \rfloor - 1$ and let 
$m_0 = \sum_{d = 1}^\ell \pi_q(d)$
be the number of irreducible polynomials in $\mathcal{M}_q$ of degree at most $\ell$.  By Lemma \ref{lem:exactsieve}, we find that the number of polynomials in $\mathcal{I}_q(0,h)$ covered by any choice of congruence classes $Q_i \bmod P_i$, for each $i\leq m_0$ is exactly
\[ q^{h+1} \bigg( 1 - \prod_{d=1}^\ell\prod_{d \in \mathcal P_q(d)} \bigg( 1 - \frac 1{q^{kd}} \bigg) \bigg). \]
This leaves us, as $h \to \infty$, with 
\begin{align*}
   m_1 \coloneqq q^{h+1} \prod_{d=1}^\ell\prod_{d \in \mathcal P_q(d)} \bigg( 1 - \frac 1{q^{kd}} \bigg) &= q^{h+1}\left(\frac{1}{\zeta_q(k)} + O\left(\frac{1}{\ell q^\ell}\right)\right) \\&= q^{h+1}\left(\zeta_q(k)^{-1}+ o\left( h^{-1} \right)\right)  
\end{align*} 
uncovered polynomials, which we cover trivially, using one congruence class for each.  Thus, the total number of congruences we require is 
\begin{align}
    m = m_0+m_1 & = m_0 + q^{h+1}\left( \zeta_q(k)^{-1} + o\left( h^{-1} \right) \right)\nonumber \\ 
    &= q^{h+1}\left( \zeta_q(k)^{-1} + o\left( h^{-1} \right) \right),
\end{align} 
after noting that $ m_0 \le \sum_{q \le \ell} \frac {q^d}d \ll q^\ell \ll h$. Using this value of $m$ in our expression for $\delta(m)$ we find that
\begin{align*}
    \delta(m) &= k\zeta_q(k)^{-1}q^{h+1}\left(h + \log_q h + c_0 +o(1) \right),
\end{align*}
where $c_0 = \log_q((q-1)/e) + 1 - \log_q \zeta_q(k)$. 
Hypothesis \eqref{eq:hbound} ensures that, 
for sufficiently large $h$ and $n$, 
\begin{align}\label{eq:delta_bnd}
\delta(m) &\le \frac {kcn}{\zeta_q(k)\log_qn}(\log_q n  + O(1)) < n,
\end{align}  
by the assumption that $kc < \zeta_q(k)$.  Thus, the polynomial that we have constructed has degree at most $n$ and the result follows.
\end{proof}

\begin{remark}
One sees readily that if hypothesis \eqref{eq:hbound} is replaced by
\[ q^{h+1} \leq \frac {\zeta_q(k)}{k} \cdot \frac {nq^{\varepsilon(n)}}{\log_q n}, \]
inequality \eqref{eq:delta_bnd} can be refined to
\[ \delta(m) \leq \frac {nq^{\varepsilon(n)}}{\log_q n}\left( \log_q n + \log_q \left( \frac {q-1}{ke} \right)  + o(1) \right). \]
In particular, when $\varepsilon(n) \leq \log_q \big(1-\frac{c'}{\log_q n}\big)$ with $c' > \log_q \left( (q-1)/ke \right)$, we find that $\delta(m) < n$.
\end{remark}

\section{Proof of the main theorem}
\label{sec:carmon}

In this section, we extend the ideas from \S\ref{sec:carmon-basic} to prove Theorem \ref{thm:main}. We finish the section with brief remarks on the proof that justify our comments in the introduction about possible enhancements to the theorem. We also remark on the conclusions one can draw when $h$ and $n$ are of moderate size and how such conclusions compare to Theorems \ref{thm:n3} and \ref{thm:n4}.

\begin{proof}[Proof of Theorem \ref{thm:main}]
Consider an integer $k \ge 2$. Similarly to \S\ref{sec:carmon-basic}, we start from \eqref{eqn:bound} and use \eqref{eqn2.3}--\eqref{eq:sig2bd} to bound the contribution to the right side of \eqref{eqn:bound} from irreducible polynomials $P$ with $\deg P \le \ell=h$. Thus, we focus on the quantity 
\[ \Sigma_3 = |\{G \in \mathcal{I}_q(F,h):  P^k \mid G \text{ for some } P \in \mathcal{P}_q, \, \deg(P) > \ell \} |. \]

As in the case $k=2$ before, we find that if $P^k$ divides $G$, then $P$ divides the first $k-1$ derivatives of $G$, and that the $j$-th derivative, $G^{(j)}$, lies in the set 
\[ \mathcal{I}_q^{(j)} (F,h) \coloneqq \mathcal{I}_q(F^{(j)},h-j) \cap \bigcap_{i + j \ge p} \mathcal D_q^i. \]

When $k < p$, we may use these observations in a similar fashion to \S\ref{sec:carmon-basic} to complete the proof. To begin, let $h = ps + r$, with $0 \le r < p$. We observe that $\mathcal I_q^{(k-1)}(F,h)$ is contained in a shift of a finite-dimensional linear space over $\mathbb F_q$ of dimension 
\begin{align*}
h+1 - \sum_{i=1}^{k-1} \left\lceil \frac {h+1-i}{p} \right\rceil  &= h+1 - (k-1)s - \min(r+1,k-1)  
 \leq (h+1) \left( 1 - \frac {k-1}p \right).
\end{align*} 
Hence,
\[ \left| \mathcal I_q^{(k-1)}(F,h) \right| \leq q^{(h+1)(1 - (k-1)/p)}. \]
By a similar counting argument, we find that there are $\le q^{(k-1)(h/p + 1)}$ polynomials $G \in \mathcal I_q(F,h)$ with $G^{(k-1)} = 0$. Next, we will show that for each of the $< n/(\ell+1)$ irreducible factors $P$, with $\deg P > \ell$, of a nonzero polynomial $H \in \mathcal{I}_q^{(k-1)}(F,h)$, there is at most one $G\in \mathcal{I}_q(F,h)$ divisible by $P^k$. From this, we can conclude that
\begin{equation}\label{eqn:s3-bnda}
\Sigma_3 \leq \frac n{\ell+1} q^{(h+1)(1-(k-1)/p)} + q^{(k-1)(h/p + 1)} .    
\end{equation}

Consider a nonzero $H \in \mathcal{I}_q^{(k-1)}(F,h)$ and an irreducible factor $P$ of $H$ of degree at least $h+1$. A polynomial $G\in \mathcal{I}_q(F,h)$ divisible by $P^k$ exists if and only if we can find a finite sequence of polynomials $H_{k-1} = H, H_{k-2}, \dots, H_1, H_{0} = G$,  each divisible by $P$, such that 
\[ H_j \in \mathcal{I}_q^{(j)}(F,h), \quad H_j'=H_{j+1} \qquad (0 \le j < k-1). \]  
Since $\deg P > h$, $\mathcal{I}_q^{(j)}(F,h)$ can contain at most one multiple of $P$, so for each $j$, there is at most one possibility for the polynomial $H_{j}$. In particular, at most one possible polynomial $G \in \mathcal{I}_q(F,h)$ is divisible by $P^k$. This establishes our earlier claim and completes the proof of \eqref{eqn:s3-bnda}.

Suppose now that $k \geq p$ and $G \in \mathcal{I}_q(F, h)$ is divisible by $P^k$ for some $P\in \mathcal{P}_q$. We intend to take $p-1$ derivatives of $G$, but we need to proceed with care. Recall the base-$p$ representation of $k$: 
\[ k = dp^a + \dots + d_1p + d_0 =: k_1p + d_0. \]
After taking $d_0$ derivatives of $G$, we have $G^{(d_0)} = P^{k_1p}Q$ for some polynomial $Q$, and afterwards we find that
\[ G^{(j)} = P^{k_1p}Q^{(j-d_0)} \quad (j \ge d_0). \]
In particular, $P^{k_1p} \mid G^{(p-1)}$. On the other hand, $G^{(p-1)} \in \mathbb F_q[x^p]$, so $G^{(p-1)} = H^p$ for some polynomial $H \in \mathbb F_q[x]$. 

When $i > (h+1)/p-1$, the coefficient of  $x^i$ in $H$ depends only on a single coefficient of $F$. So $H \in \mathcal I_q(F_1, h_1)$, where $h_1 = \lfloor (h{+}1)p^{-1} \rfloor -1$ and $F_1=(F^{(p-1)})^{1/p}$ is a polynomial of degree $< np^{-1}$ determined uniquely by $F$. Moreover, by the uniqueness of polynomial factorization in $\mathbb F_q[x]$, we have $P^{k_1} \mid H$. On the other hand, if $H \in \mathcal I_q(F_1, h_1)$ is nonzero and divisible by $P^{k_1}$ for some irreducible polynomial $P$, with $\deg P > \ell$, the argument we gave to justify \eqref{eqn:s3-bnda} shows that there is at most one polynomial $G \in \mathcal I_q(F,h)$ such that $G^{(p-1)} = H^p$ (and $G, G', \dots, G^{(p-1)}$ are all divisible by $P$). 

Let $\mathcal S_1 \subset \mathcal I_q(F_1, h_1)$ be a set (with $0 \in \mathcal S_1$ if $0\in \mathcal I_q(F_1, h_1)$) to be specified shortly. For any $H \in \mathcal S_1$, there are $q^{h-h_1}$ polynomials $G \in \mathcal I_q(F,h)$ with $G^{(p-1)} = H^p$, so we find that
\begin{equation}\label{eqn:s3-1}
\Sigma_3 \le \Sigma_{3,1} + q^{h-h_1}|\mathcal S_1|, 
\end{equation} 
where $\Sigma_{3,1}$ counts pairs $(H,P)$, with $P \in \mathcal P_q$, $H \in \mathcal I_q(F_1, h_1) \setminus \mathcal S_1,$ $P^{k_1} \mid H$, and $\deg P > \ell$.

When $k_1 \ge p$ (equivalently $a \ge 2$), we can iterate the above argument, with a slight twist. If $(H_1,P)$ is one of the pairs counted by $\Sigma_{3,1}$, the above construction with $H_1$ in place of $G$ yields a polynomial $H_2 \in \mathcal{I}_q(F_2,h_2)$, where $h_2 = \lfloor (h_1+1)p^{-1} \rfloor -1$ and $F_2$ a polynomial of degree $<np^{-2}$ determined uniquely by $F_1$ (and therefore, by $F$). Moreover, we have $P^{k_2} \mid H_2$, where $k_2 = (k_1-d_1)/p$. Suppose now that 
$\mathcal S_2 \subset \mathcal I_q(F_2, h_2)$ is a set of polynomials, to be specified shortly (with $0 \in \mathcal S_2$ if $0 \in \mathcal I_q(F_2, h_2)$. 

We now specify $\mathcal S_1$ as the set of polynomials $H \in \mathcal I_q(F_1, h_1)$ such that $H^{(p-1)} = A^p$ for some $A \in \mathcal S_2$ (note that this condition ensures that $0 \in \mathcal S_1$ if $0 \in \mathcal I_q(F_1, h_1)$). For each $A \in \mathcal S_2$, there are $\leq q^{h_1-h_2}$ polynomials $H \in \mathcal S_1$, so 
\begin{equation}\label{eqn:s3-1a}
 |\mathcal S_1| \leq q^{h_1 - h_2}|\mathcal S_2|.  
\end{equation} 
For any such choice of $\mathcal S_2$, we find that $ \Sigma_{3,1} \le \Sigma_{3,2}$,  
where $\Sigma_{3,2}$ counts pairs $(H,P)$, with $P$ irreducible, $H \in \mathcal I_q(F_2, h_2) \setminus \mathcal S_2$, $ P^{k_2} \mid H$,  and $\deg P > \ell$. Therefore, we deduce that
\begin{equation}\label{eqn:s3-2}
\Sigma_3 \le \Sigma_{3,2} + q^{h-h_2}|\mathcal S_2|. 
\end{equation} 

In general, we can iterate this argument a total of $a$ times to find a polynomial $F_a$ of degree $<np^{-a}$,  determined uniquely by $F$, such that
\[ \Sigma_{3} \le \Sigma_{3,a} + q^{h - h_a}|\mathcal S_a|, \]
where $h_a =\lfloor (h_{a-1}+1)p^{-1}\rfloor-1$, the set $\mathcal S_a \subset \mathcal I_q(F_a, h_a)$ is at our disposal to choose (so long as $0\in \mathcal{S}_a$ if $0 \in \mathcal I_q(F_a, h_a)$), and $\Sigma_{3,a}$ is the number of pairs $(H,P)$, with $P$ irreducible, subject to
\[ H \in \mathcal I_q(F_a, h_a) \setminus \mathcal S_a, \quad P^{d} \mid H, \quad \deg P > \ell. \]
A short computation shows that \begin{equation} \label{eq:habds}
    (h+1)p^{-a} - 3 \leq h_a \leq (h+1)p^{-a} - 1.
\end{equation} 
At this point, we choose $\mathcal S_a$ to be the set of polynomials $H \in \mathcal I_q(F_a, h_a)$ with $H^{(d-1)} = 0$, so $|\mathcal S_a| \leq q^{(d-1)(h_ap^{-1}+1)}$. 
Hence, 
\begin{equation}\label{eqn:s3-a}
 \Sigma_{3} \le \Sigma_{3,a} + q^{h - h_a + (d-1)(h_a/p+1)}. 
\end{equation}

When $d=1$, we apply the trivial bound for $\Sigma_{3,a}$:
\begin{equation} \label{eq:s3-ad=1}
    \Sigma_{3,a} \le (\deg F_a)(\ell+1)^{-1}|\mathcal I_q(F_a, h_a)| < \frac {n}{p^a(\ell+1)} q^{(h+1)p^{-a}}. 
\end{equation} 
When $d > 1$, we may bound $\Sigma_{3,a}$ using a variant of \eqref{eqn:s3-bnda} with $h = h_a$, $k = d$, and $n = \deg F_a$. Recall that the second term on the right side of \eqref{eqn:s3-bnda} accounts for polynomials in $G \in \mathcal I_q(F,h)$ with $G^{(d-1)} = 0$. Thus, by our choice of $\mathcal S_a$, the respective bound for $\Sigma_{3,a}$ becomes
\[ \Sigma_{3,a} \le (\deg F_a)(\ell+1)^{-1}q^{(h_a+1)(1-(d-1)/p)} < \frac {n}{p^a(\ell+1)} q^{(h+1)p^{-a}(1-(d-1)/p)}. \]
Note that setting $d=1$ in the bound above yields the exact same expression as \eqref{eq:s3-ad=1}.  
So, in either case, using this in \eqref{eqn:s3-a} along with \eqref{eq:habds} yields
\begin{equation}\label{eqn:s3-a2}
 \Sigma_{3} \le \frac {n}{p^a(h+1)} q^{(h+1)p^{-a}(1-(d-1)/p)} + q^{(h+1)(1 - (p-d+1)p^{-a-1}) + d+1}. 
\end{equation}

Note that when $a = 0$ and $d = d_0 = k < p$, \eqref{eqn:s3-bnda} is a slightly stronger version of \eqref{eqn:s3-a2} (whose second term contains an extra factor of $q^{2+(k-1)/p} < q^3$). Therefore, 
we combine \eqref{eqn:s3-a2} with \eqref{eqn:bound} and \eqref{eqn2.3}--\eqref{eq:sig2bd} to conclude, for sufficiently large $h$ and any $k\geq 2$, that
\begin{align}\label{eqn:N-bnd1}
\mathcal{N}_{q}(F,h) &\leq q^{h+1}\left( \ln \zeta_q(k) + \frac {n}{p^{a}(h+1)} q^{-\theta(h+1)} + O\left( h^{-1} \right) \right),
\end{align}
where $\theta = 1-(p-d+1)p^{-a-1}$.

If $h$ is chosen so that 
\begin{align}\label{eqn:h-cond} 
q^{h+1} \ge \left( \frac {cn}{\log_q n} \right)^{1/\theta} 
\end{align}
for some absolute constant $c > 0$, it follows that
\begin{align}\label{eqn:N-bnd2}
\mathcal{N}_{q}(F,h) &\leq q^{h+1}\left( \ln\zeta_q(k) + \frac{\theta}{p^ac} + O\left( \frac {\log_q h}{h} \right) \right).
\end{align}
Since $\theta \le 1 - k^{-1}$, this establishes the theorem for $c \ge 1$.    
\end{proof}

\begin{remark}
Suppose that $n$ and $h$ are large, and let $\Sigma_1'$, be the subsum of $\Sigma_1$ (in \eqref{eqn2.3}) with $h_0 < d \le h$, where $h_0 = o(h)$. Also, let 
\[ \Sigma_0 = \left| \left\{ Q \in \mathcal I_q(F,h) : P^k \mid Q \text{ for some } P \in \mathcal P_q, \; \deg P \le h_0 \right\} \right|. \]
Choosing $h_0$ sufficiently small in terms of $h$, one may apply a sieve argument (similar to the proof of Lemma \ref{lem:exactsieve}) to $\Sigma_0$ to obtain an asymptotic formula for $\Sigma_1$. One can then replace the term $\ln\zeta_q(k)$ in \eqref{eqn:N-bnd2} by $1-\zeta_q(k)^{-1}+o(1)$.  From this we see that $c$ can be taken to be any constant $c> \theta \zeta_q(k)p^{-a}$ so long as $n$ is taken sufficiently large. 

On the other hand, if the constant $c$ in \eqref{eqn:h-cond} is replaced by a function $g(n) \to \infty$ as $n \to \infty$, one may turn the above bounds into an asymptotic formula for the number $\mathcal Q_q(F,h)$ of $k$-free polynomials in $\mathcal I_q(F,h)$, since 
\[ \Sigma_0 - \Sigma_1' - \Sigma_2 - \Sigma_3 \leq \mathcal Q_q(F,h) \leq \Sigma_0. \]
\end{remark}

\begin{remark}

Theorems \ref{thm:n3} and \ref{thm:n4} give fully explicit ranges of $q$ and $n$ for which the short interval $\mathcal I_q(F,h)$ contains $k$-free polynomials under the respective constraints on $h$, because bounds like  \eqref{eq:n/3-bnd} above (see also \eqref{eq:n3-ch2}, \eqref{eq4.3}, and \eqref{eqn:n4-chk} in \S\ref{sec:spacings}) are explicit. Theorem \ref{thm:main}, on the other hand, is stated for sufficiently large $n$ to simplify the analysis of \eqref{eqn:N-bnd1}, which focuses on the case when $h$ and $n$ are large. However, if the contributions to the right side of \eqref{eqn:N-bnd1} from $\Sigma_2$ and $\Sigma_3$ are kept explicit, one can determine, for every fixed triple $(k,q,h)$, with $h \ge h_0(k,q)$, a range of degrees $n$ for which $\mathcal I_q(F,h)$ contains $k$-free polynomials. It appears difficult to channel such observations into a general statement similar to Theorems \ref{thm:n3} and \ref{thm:n4}, but it is possible to draw on them to gain some broad insights.

For example, Table \ref{table1} lists several pairs $(q,h)$ and the values of respective integers $n_0(q,h)$ such that the interval $\mathcal I_q(F,h)$ contains a squarefree polynomial whenever $\deg F \le n_0$. For values of $q$ with $p>2$, these bounds are computed using \eqref{eq:sqfreeintbd}, taking the largest value of $n$ such that the coefficient of $q^{h+1}$ is less than 1.  For those values with $p=2$, an explicit version of \eqref{eqn:N-bnd1} is used, after noting that in this specific case, $k=p=2$, the lower bound in \eqref{eq:habds} can be improved to $(h+1)/2-3/2$. This results in an expression identical to \eqref{eq:sqfreeintbd}, but in which the second to last term is half as large.  Note that even these values are likely much smaller than the ``truth.'' For example, in the case $q=2$ these methods do not prove that all short intervals with $h=1$ or 2 and any value of $n$ are guaranteed to contain squarefree polynomials, however direct computation shows that every such short interval contains a squarefree polynomial in these cases when $n\leq 9$ and 16 respectively.

In some cases, the bounds obtained using Theorems \ref{thm:n3} and \ref{thm:n4} are stronger than those obtained here.  Such improved bounds are included in the table above marked with the symbols * or \dag.  

\begin{table}
\begin{center}
\begin{tabular}{cccccccccccc}
\toprule
    $h$\ \textbackslash\  $q$          & 2 & 3   & 4   & 5 & 7 & 8 & 9 & 11 & 19  & 25 & 27 \\
\midrule
$h=1$       & --- & $3^*$ & $3^*$ & $3^*$ & $4^\dag$ & 11 & 4 & $4^\dag$ & {$4^\dag$} & {6} & 14 \\
$h=2$       & --- & $6^*$ & 12 & $6^*$ & $8^\dag$ & 89 & 20 & $8^\dag$ & $8^\dag$& 19 & 75  \\
$h=3$       & --- & $9^*$ &  57 & 9 & $12^\dag$ & 393 & 61 & $12^\dag$ & $12^\dag$ &  49 & 307 \\
$h=4$       & --- & $12^*$ & 174 & 17 & $16^\dag$ & 1467 & 164 & $16^\dag$ & $16^\dag$ &  118 & 1156 \\
$h=5$       & --- & 23 & 459 & 29 & 25 & 5092 & 414 & $20^\dag$ & $20^\dag$ &  271 & 4173 \\
$h=6$       & --- &  42 & 1124 & 48 & 39 & 16984 & 1013 & 28 & $24^\dag$ &  603 & 14629 \\
$h=7$       & --- & 73 & 2641 & 77 & 60 & 55234 & 2417 & 40 &  $28^\dag$ &  1314 & 50207 \\
$h=8$       & 23 & 123 & 6048 & 120 & 90 & 176448 & 5674 & 57 & 34 & 2818 & 169578 \\
\bottomrule
\end{tabular}
\bigskip
\caption{Values of $n_0(q,h)$ for which $\mathcal I_q(F, h)$ contains a squarefree polynomial whenever $\deg F \le n_0$.  Numbers marked with an * or a \dag\  were obtained using Theorems \ref{thm:n3} or \ref{thm:n4}, respectively.} 
\label{table1}    
\end{center}  
\end{table}

\end{remark}

\section{The differencing method for polynomials}
\label{sec:spacings}

Recall the set $\mathcal S_q(d)$ defined in \eqref{def:Td}. In this section, we prove several results about the spacing between elements of $\mathcal S_q(d)$. Through applications of Lemma \ref{lem2.3}, these results will then yield upper bounds on $|\mathcal S_q(d)|$, which apply to prove Theorems \ref{thm:n3} and \ref{thm:n4}. 

Our first result is a bound on the minimum degree of the difference of distinct elements of $\mathcal S_q(d)$. Recall that we use $p$ to denote the characteristic of the finite field $\mathbb F_q$. We note that when $p \nmid k$, we have $r = 1$ in the proposition below, while when $p \mid k$, we have $1 < r \le k$.  

\begin{proposition}\label{prop1}
Suppose that $h < d \le n/k$ and $G, H \in \mathcal S_q(d)$, with $G \neq H$. Let $r = r(k,p)$ be the least positive integer such that $p \nmid \binom kr$. When $r < k$, we have
\begin{equation}\label{eqn:5.1}
\deg(G-H) \geq \frac {(k+r)d-n}{r}.
\end{equation}
When $r = k$, we have either \eqref{eqn:5.1} with $r = k$, or 
\begin{equation}\label{eqn:5.3} 
\deg(G-H) \le \frac {h + kd - n}k. 
\end{equation}
\end{proposition}

\begin{proof}
Let $A,B \in \mathcal M_q(n-kd)$ be such that $G^kA, H^kB \in \mathcal I_q(F,h)$. Then $\deg(G^kA - H^kB)\leq h$, and we deduce that
\begin{equation}\label{eqn:prop1.1} 
\deg\big((G^k-H^k)A+H^k(A-B)\big) \leq h.
\end{equation}
Note that
\[ G^k - H^k = \big((G-H)+H\big)^k -H^k= \sum_{j=1}^k \binom {k}{j}(G-H)^jH^{k-j}. \] 
Since $\deg(G-H) < d = \deg H$, it follows that
\begin{equation}\label{eqn:prop1.2} 
\deg(G^k - H^k) = r\deg(G-H) + (k-r)d. 
\end{equation}

Suppose first that $A \neq B$. Then the degree of the second term on the left side of  \eqref{eqn:prop1.1} is
\[ \deg( H^k(A-B) ) = kd + \deg(A - B) \geq kd > h. \]
This is only possible if the two terms on the left side of \eqref{eqn:prop1.1} have the same degree, meaning that
\begin{equation*}
    \deg((G^k - H^k)A) = \deg (H^k(A-B)) \geq kd. \label{eq:sqdifflowbd}
\end{equation*}   
Combining this with \eqref{eqn:prop1.2} gives
\begin{align*} 
kd &\le r\deg(G-H) + (k-r)d + \deg A = r\deg(G-H) + (n-rd), 
\end{align*}
which establishes \eqref{eqn:5.1} in this case. 

Next, we consider the case $A = B$. Then \eqref{eqn:prop1.1} and \eqref{eqn:prop1.2} give
\[ r\deg(G-H) + (n-rd) = \deg((G^k - H^k)A) \leq h, \]
and hence,
\[ \deg(G-H) \le \frac {h + rd - n}{r}. \]
When $r = k$, this establishes \eqref{eqn:5.3}, and when $r < k$, we get
\[ \deg(G-H) < \frac{(r+1)d - n}r \le 0, \]
which contradicts our assumption that $G \neq H$. Therefore, this case occurs only when $r = k$.
\end{proof}

We remark that when $r = k$ and $d < (n-h)/k$, inequality \eqref{eqn:5.3} contradicts the assumption $G \neq H$ of the proposition, so for $d$ in this range, we always have \eqref{eqn:5.1}. On the other hand, when $r=k$ and $(n-h)/k \le d \le n/k$, we can combine \eqref{eqn:5.1} and \eqref{eqn:5.3} to obtain a rather sharp bound on $|\mathcal S_q(d)|$, which we state in the following lemma.

\begin{lemma}\label{lm:p2-n3}
Assume the notation of Propositon \ref{prop1}. If $r=k$ and $(n-h)/k \le d \le n/k$, we have
\[ |\mathcal S_q(d)| \le q^{h/k+1}.  \]
\end{lemma}

\begin{proof}
Let $\delta = (2kd-n)/k$. According to the proposition, any two elements $G, H$ of $\mathcal S_q(d)$ with $\deg(G-H) < \delta$ must satisfy \eqref{eqn:5.3}. In particular, for a fixed $G$, there are at most 
\[ |\mathcal I_q(G, (kd+h-n)/k)| \le q^{(kd+h-n)/k + 1} =: \kappa \]
polynomials $H \in \mathcal S_q(d)$ with $\deg(G-H) < \delta$. Thus, Lemma \ref{lem2.3} gives 
\[ |\mathcal S_q(d)| \le \kappa q^{d-\delta} = q^{h/k + 1}.  \qedhere \]
\end{proof}

Recall that in \S\ref{sec:prelim} we derived Theorem \ref{thm:n3} in the case when $p \nmid k$ from \eqref{eqn:5.1} with $r = 1$. We can use Lemma \ref{lm:p2-n3} to complete the proof of Theorem \ref{thm:n3} in the case when $p \mid k$.

\begin{proof}[Proof of Theorem \ref{thm:n3}: $p \mid k$]
To begin, we observe that when $d \le n/k$ and $r > 1$ in Proposition~\ref{prop1}, the bound \eqref{eqn:5.1} is stronger than its version with $r=1$. Therefore, when $1 < r < k$, we still have inequality \eqref{bd1} (and more), and so we may follow the proof from the case $p \nmid k$ (given in \S\ref{sec:2.1}) without any changes. Thus, we may focus on the case $r = k$.

Note that when $h \ge n/(k+1)$, Lemma \ref{lm:p2-n3} is applicable in the full range $h < d \le n/k$. Hence,
\begin{equation}\label{eq:n3-sig4} 
\sum_{h < d \le n/k} |\mathcal S_q(d)| \leq \left( \frac nk - h \right) q^{h/k+1} \leq (h/k) q^{h/k+1}.     
\end{equation}
Combining this with \eqref{eqn:bound} and \eqref{eqn2.3}--\eqref{useTd} with $\ell = h $, we find that
\begin{equation}\label{eq:n3-ch2}
\mathcal N_q(F,h) \leq q^{h+1} \left( \ln \zeta_q(k) + \frac {q+k_h-1}{q(q-1)h} + (h/k) q^{(1-k)h/k} \right).     
\end{equation}
When $k \ge 3$, the last expression is $< q^{h+1}$, provided for all $q \ge 3$ and $h \ge 1$. When $k=2$ (note that in this case, we have $p=2$ and $q = 2^f$), the same holds for $q \ge 4$ and $h \ge 1$. 
\end{proof}

Next, we consider $s$-tuples of distinct polynomials $\mathbf G = \{G_1, \dots, G_s \}$ in $\mathcal S_q(d)$, with $s \ge 3$. If $\mathbf G$ is such an $s$-tuple, we write
\begin{gather*}
\delta(\mathbf G) = \min_{1\leq i<j\leq s} \deg(G_i-G_j),\\
\Delta(\mathbf G) = \max_{1\leq i<j\leq s} \deg(G_i-G_j).
\end{gather*} 
By Proposition~\ref{prop1}, we have  
\begin{equation}\label{eqn:delta}
\delta(\mathbf G) \ge (k+1)d - n ,   
\end{equation}
whenever $h < d \le n/k$ and $r = r(k,p) < k$ (or $r = k$ and $d < (n-h)/k$). Our next result is a lower bound on $\Delta(\mathbf G)$ for triples. 

\begin{proposition}\label{prop:4d-n/3}
Suppose that $p \nmid k(k+1)$ and $h < d \le n/k$. If $\mathbf G = \{ G_1,G_2,G_3 \}$ is a set of distinct polynomials in $\mathcal S_q(d)$, then
\begin{equation} \label{eqn:DeltaG} \Delta(\mathbf G) \geq \frac{(k+2)d-n}{3}. 
\end{equation}
\end{proposition}

\begin{proof} 
For each $1\leq i \leq 3$, let $A_i\in\mathcal M_q(n-kd)$ and $R_i$ be polynomials such that 
\begin{equation} \label{eqn:5.5}
F = G_i^kA_i - R_i, \quad \deg R_i \le h. 
\end{equation}
We now consider the rational function 
\[ \Phi[G_1, G_2, G_3] =(G_3-G_2)\frac{F}{G_1^k}+(G_1-G_3)\frac{F}{G_2^k}+(G_2-G_1)\frac{F}{G_3^k}, \] 
essentially a second divided difference of the function $\Phi(t) = Ft^{-k}$ on $\mathbb F_q(x)$ (see \cite{IsKe94} for background on divided differences). By~\eqref{eqn:5.5}, we have
\begin{equation}\label{eqn:5.5a}
\Phi[G_1, G_2, G_3] = N-\Theta, 
\end{equation}
where 
\begin{align*} 
N &= (G_3-G_2)A_1 + (G_1-G_3)A_2+(G_2-G_1)A_3, \\
\Theta &=(G_3-G_2)\frac{R_1}{G_1^k}+(G_1-G_3)\frac{R_2}{G_2^k}+(G_2-G_1)\frac{R_3}{G_3^k}.
\end{align*}
Our immediate goal is to show that $N$ is a nonzero polynomial by showing that $\deg N \geq 0$. In the rest of the proof, we suppress the dependence on $\mathbf G$ and write simply $\Delta, \delta$, and $\Phi$ instead of $\Delta(\mathbf G), \delta(\mathbf G)$, and $\Phi[G_1,G_2,G_3]$. 

We can rewrite the definition of $\Phi$ as the polynomial identity
\begin{equation}\label{eqn:5.7}
\Phi G_1^kG_2^kG_3^k = F\left( \prod_{i < j}(G_j-G_i) \right) \left( \sum_{a+b+c = 2k-2} G_1^aG_2^bG_3^c \right),
\end{equation}
where the product on the right is over all pairs of indices $i,j$ with $1 \le i < j \le 3$, and the sum is over all triples $a,b,c$ with $0 \le a,b,c \le k-1$ and $a+b+c = 2k-2$. Observe that if both $\deg(G_j-G_i) < \Delta$ and $\deg(G_k - G_i) < \Delta$, then also
\[ \deg(G_j - G_k) \leq \max\{\deg(G_j - G_i), \deg(G_k - G_i) \} < \Delta, \]
which contradicts the choice of $\Delta$. Thus, at least two of the differences in the above product must have degree $\Delta$, and we get 
\begin{equation}\label{eqn:5.8}
2\Delta + \delta \leq \deg\left(\prod_{i < j}(G_j-G_i) \right) \leq 3\Delta.
\end{equation}
Also, the sum on the right side of \eqref{eqn:5.7} has $\binom {k+1}2$ terms, each of them in $\mathcal M_q((2k-2)d)$. Since $p \nmid \binom {k+1}2$, it follows that the sum is a polynomial of degree $(2k-2)d$, and we deduce 
\begin{equation}\label{dd-bd}
n +(2k-2)d + 2\Delta+ \delta \leq \deg(\Phi \, G_1^kG_2^kG_3^k)\leq n + (2k-2)d + 3\Delta.
\end{equation}

On the other hand, we have
\begin{equation}\label{eqn:5.11}
\begin{split}
\Theta G_1^kG_2^kG_3^k  = R_1(G_3-G_2)G_2^kG_3^k &+ R_2(G_1-G_3)G_1^kG_3^k + R_3(G_2-G_1)G_1^kG_2^k.
\end{split}
\end{equation}
Since each of the three terms on the right side of \eqref{eqn:5.11} has degree $\le h + \Delta + 2kd$, we obtain
\begin{equation}\label{theta-bd}
\deg \left(\Theta G_1^kG_2^kG_3^k \right) \leq 2kd + \Delta + h.
\end{equation}
Moreover, since $p \nmid k$, we have \eqref{eqn:delta} by Proposition \ref{prop1}. Combining \eqref{eqn:delta}, \eqref{dd-bd}, and \eqref{theta-bd}, we conclude that
\begin{align*} 
\deg(\Phi G_1^kG_2^kG_3^k) &\ge n + (2k-2)d + 2\Delta + \delta \\
&\ge (3k-1)d + 2\Delta \\
&> 2kd + h + 2\Delta \ge \deg \left(\Theta G_1^kG_2^kG_3^k \right).
\end{align*}
Thus, by \eqref{eqn:5.5a},
\[ \deg(N G_1^kG_2^kG_3^k) = \deg(\Phi G_1^kG_2^kG_3^k) \ge 0, \]
which establishes our prior claim that $N \neq 0$. Using the upper bound in \eqref{dd-bd}, we get
\[ 3kd \le \deg(N G_1^kG_2^kG_3^k) = \deg(\Phi G_1^kG_2^kG_3^k) \leq n+ (2k-2)d + 3\Delta, \]
and the desired conclusion follows.
\end{proof}

We now use Proposition \ref{prop:4d-n/3} and Lemma \ref{lm:p2-n3} to prove Theorem~\ref{thm:n4}. 

\begin{proof}[Proof of Theorem \ref{thm:n4}]
Suppose first that $p \nmid k(k+1)$. When $d > h$, we have $d \geq (n+1)/(k+2)$, and Proposition \ref{prop:4d-n/3} allows us to apply Lemma~\ref{lem2.3} with $\kappa = 2$ and $\delta = ((k+2)d-n)/3$ to deduce the bound
\begin{equation}
|\mathcal S_q(d)| \leq 2q^{(n-(k-1)d)/3}. \label{eq4.2}    
\end{equation} 
Therefore, 
\begin{align*}
\sum_{h < d \leq n/k} |\mathcal S_q(d)| &\leq 2\sum_{h < d \leq n/k} q^{(n-(k-1)d)/3} 
< 2q^{(n-(k-1)d_0)/3} \sum_{j \ge 0} q^{-(k-1)j/3} = \frac {2q^{(n+1)/(k+2)}}{q^{1/3}-q^{-(k-2)/3}}, 
\end{align*}
where $d_0 = (n+1)/(k+2)$.
This inequality, \eqref{eqn:bound}, and \eqref{eqn2.3}--\eqref{useTd} with $\ell = h$ now give
\begin{equation} 
\mathcal N_q(F,h) \leq q^{h+1} \left( \ln \zeta_q(k) + \frac {q+k_h-1}{q(q-1)h} + \frac {2q^{-(k+1)/(k+2)}}{q^{1/3}-q^{-(k-2)/3}}\right), \label{eq4.3}    
\end{equation}
which implies \eqref{eqn2.0} when $k=2$ and $q \ge 7$ or when $k \ge 3$ and $q \ge 5$.

Next, let $p \mid k$ and suppose that $r < k$ in Proposition \ref{prop1}. Then Proposition 1 yields  \eqref{eqn:5.1} with $r \ge 2$, and hence, with $r=2$. Thus, Lemma~\ref{lem2.3} with $\kappa = 1$ and $\delta = ((k+2)d-n)/2$ yields the bound
\begin{equation} 
|\mathcal S_q(d)| \le q^{(n-kd)/2}, \label{eq4.2a}    
\end{equation}
which supersedes \eqref{eq4.2}. Hence,
\eqref{eq4.3} holds also in this case.

Finally, let $p \mid k$ and $r = k$, and assume that $n/(k+2) \le h < n/(k+1)$. When $h < d < (n-h)/k$, we can again use Proposition \ref{prop1} to obtain \eqref{eq4.2a}. 
Hence,
\begin{equation}\label{eqn:n4-sig3}
\begin{split}
\sum_{h < d < (n-h)/k} |\mathcal S_q(d)| &\le q^{(n-kd_0)/2} \sum_{j \ge 0} q^{-kj/2} = \frac {q^{(n+1)/(k+2)}}{q^{1/2}-q^{-(k-1)/2}}.
\end{split}
\end{equation}
Moreover, when $(n-h)/k \le d \le n/k$, we may use Lemma \ref{lm:p2-n3} in a similar fashion to \eqref{eq:n3-sig4} to show that
\begin{equation}\label{eqn:n4-sig4}
\sum_{(n-h)/k \le d \le n/k} |\mathcal S_q(d)| \leq \left( \frac hk  + 1 \right) q^{h/k+1}.
\end{equation}
Combining \eqref{eqn:n4-sig3} and \eqref{eqn:n4-sig4} with \eqref{eqn:bound} and \eqref{eqn2.3}--\eqref{useTd} with $\ell = h$, we conclude that
\begin{equation}\label{eqn:n4-chk} \mathcal N_q(F,h) \leq q^{h+1} \left( \ln \zeta_q(k) + \frac {q+k_h-1}{q(q-1)h} + \frac {q^{-(k+1)/(k+2)}}{q^{1/2}-q^{-(k-1)/2}} + \left( \frac hk + 1 \right) q^{(1-k)h/k} \right), 
\end{equation}
which again implies \eqref{eqn2.0}.
\end{proof}

\appendix
\section{An analogue of the methods of Halberstam and Roth}

As in the proofs of Theorems \ref{thm:n3} and \ref{thm:n4}, we need to estimate
\[ \Sigma_3 = \sum_{\ell < d \leq n/k}|\mathcal S_{q}(d)|, \]
where $\ell \ge h$. When $k \ge 3$, the estimation of $\Sigma_3$ relies on the following proposition.

\begin{proposition}\label{prop4}
Let $k \ge 3$ and $p \nmid k \binom{2k-1}{k-1}$. If $n/(2k) \le h < d \le n/k$, we have 
\[ |\mathcal S_{q}(d) |\leq 2kq^{(n-d)/(2k-1)}.\]
\end{proposition}

We postpone the proof of this result until the end of the section and focus first on the proof of Theorem \ref{thm:k}. By the proposition,
\begin{align*}
\Sigma_3 &\leq \sum_{\ell < d \le n/k} 2kq^{(n-d)/(2k-1)} < 2kq^{(n-\ell)/(2k-1)}\sum_{j \ge 0} q^{-j/(2k-1)}\\
&= \frac {2kq^{(n-\ell)/(2k-1)}}{1-q^{-1/(2k-1)}} \le  \frac {2kq^{h + (h-\ell)/(2k-1)}}{1-q^{-1/(2k-1)}},
\end{align*}
on recalling that $h \ge n/(2k)$. Writing $\delta_q = q^{-1/(2k-1)}$, we have
\[ 1 - \delta_q  > \frac {1 - \delta_q^{2k-1}}{2k - 1} = \frac {q-1}{(2k - 1)q}, \]
so
\begin{equation}\label{eqn:8d}
\Sigma_3 \leq \frac {2k(2k-1)q^{h+1}h^{(h-\ell)/(2k-1)}}{q - 1}.
\end{equation} 

Together, \eqref{eqn:bound}, \eqref{eqn2.3}--\eqref{useTd}, 
and \eqref{eqn:8d} give
\[ \mathcal N_{q}(F,h) < q^{h+1} \left( \ln\zeta_q(k) + \frac {(q+k)q^{\ell-h}}{q(q-1)h} + \frac {k(4k-2)q^{(h-\ell)/(2k-1)}}{q-1} \right). \]
We now select 
\[ \ell = h + \log_q (qh)^{(2k-1)/2k}. \]
This choice essentially balances the second and third terms on the right side of the last inequality and gives
\begin{align} \label{eqn:nk-final}
\mathcal N_{q}(F,h) < q^{h+1} \left( \ln\zeta_q(k) + \frac {(q+4k^2-k)(qh)^{-1/2k}}{q-1} \right).    
\end{align}
When $h$ is sufficiently large in terms of $k$, this completes the proof of the theorem. \\

All that remains is to prove Proposition \ref{prop4}.
\begin{proof}[Proof of Proposition \ref{prop4}]
Consider the polynomials $P_0,Q_0\in\mathbb Z[x]$ given by  
\begin{gather*}
    P_0(x)=1-\binom{2k-1}{1}x+\dots+(-1)^{k-1}\binom{2k-1}{k-1}x^{k-1},\\
    (1-x)^{2k-1}=P_0(x)-x^k Q_0(x).
\end{gather*}
We use these to define the degree-$(k-1)$ forms 
\[ P(x,y)=x^{k-1}  P_0\left( y/x \right), \quad Q(x,y)=x^{k-1}  Q_0\left( y/x \right), \]
which satisfy the algebraic identity
\[ (x-y)^{2k-1} = x^k P(x,y) - y^k Q(x,y). \]
In particular, for any polynomials $G_1, G_2 \in \mathbb F_q[x]$, we obtain
\begin{equation}\label{eqn:8.1} 
(G_1-G_2)^{2k-1} = G_1^k P(G_1,G_2) - G_2^k Q(G_1,G_2).
\end{equation}

Next, we consider \eqref{eqn:8.1} when $G_1, G_2 \in \mathcal S_{q,k}(d)$. We find polynomials  $A_i \in \mathcal M_q(n - kd)$ and $R_i$ with
\begin{equation}\label{eqn:k1} 
F = G_i^kA_i - R_i, \quad \deg R_i \le h. 
\end{equation}
We may then rearrange \eqref{eqn:8.1} as
\begin{align} \label{eqn:8.2}
(G_1-G_2)^{2k-1}F &= F \big( G_1^k P(G_1,G_2) - G_2^k Q(G_1,G_2) \big) = N + \Theta, 
\end{align}
where
\begin{align*}
N&= G_1^kG_2^k\big( P(G_1,G_2)A_2 - (G_1,G_2)A_1 \big), \\
\Theta&= G_2^kQ(G_1,G_2)R_1 - G_1^kP(G_1,G_2)R_2. 
\end{align*}
Note that
\[ \deg \Theta \le (2k-1)d + h < 2kd. \]
When
\begin{equation}\label{eqn:8.3} 
 \deg(G_1-G_2) < \frac {2kd - n}{2k-1} =: \Delta_k, 
\end{equation}
we find also that
\[ \deg\left( (G_1-G_2)^{2k-1}F \right) = (2k-1)\deg(G_1-G_2) + n < 2kd. \]
Thus, under condition \eqref{eqn:8.3}, we can deduce from \eqref{eqn:8.2} that
\[ \deg\big( G_1^kG_2^k (  P(G_1,G_2)A_2 -  Q(G_1,G_2)A_1 ) \big) < 2kd. \]
Since $\deg(G_1^kG_2^k) = 2kd$, this is possible only if 
\begin{equation}\label{eqn:k4}
 P(G_1,G_2)A_2 - Q(G_1,G_2)A_1 = 0. 
\end{equation}
That is, if $G_1,G_2 \in \mathcal S_{q}(d)$ satisfy \eqref{eqn:8.3}, then $G_1,G_2$, and the respective polynomials $A_1,A_2$ must satisfy the polynomial identity \eqref{eqn:k4}. 

Consider a third polynomial $G_3 \in \mathcal S_{q}(d)$ such that
\begin{equation}\label{eqn:8.5} 
 \deg(G_3-G_i) < \Delta_k
\end{equation}
holds for $i=1$. Then, as an immediate consequence of \eqref{eqn:8.3}, \eqref{eqn:8.5} holds also for $i=2$. Further, by the argument in the last paragraph, we have also
\begin{equation}\label{eqn:k5} 
 P(G_1,G_3)A_3 - Q(G_1,G_3)A_1 = 0, 
\end{equation}
and
\begin{equation}\label{eqn:k6} 
 P(G_3,G_2)A_2 - Q(G_3,G_2)A_3 = 0. 
\end{equation}
Finally, from \eqref{eqn:k5} and \eqref{eqn:k6}, we readily obtain that
\begin{equation}\label{eqn:k7}  
P(G_1,G_3)P(G_3,G_2)A_2 - Q(G_1,G_3)Q(G_3,G_2)A_1 = 0. 
\end{equation}

We now consider an interval $\mathcal I$ of length $\le \Delta_k$ and fix distinct polynomials $G_1, G_2 \in \mathcal S_{q}(d) \cap \mathcal I$. Then $G_1, G_2$ satisfy \eqref{eqn:8.3}, and any other polynomial $G_3 \in \mathcal S_{q}(d) \cap \mathcal I$ must satisfy \eqref{eqn:k7}. We view 
\begin{equation}\label{eqn:k8}  
P(G_1,t)P(t,G_2)A_2 - Q(G_1,t)Q(t,G_2) A_1 = 0
\end{equation}
as a polynomial equation in $t$ over $\mathbb F_q[x]$. By the construction of $P$ and $Q$, the left side of \eqref{eqn:k8} is a polynomial of degree $2k-2$ with leading coefficient
\[ (-1)^{k-1}\binom {2k-1}{k-1}(A_2 - A_1). \]
We will show that this coefficient is nonzero. The hypothesis on the characteristic $p$ reduces this task to showing that $A_1 \neq A_2$. 

When $A_1 = A_2 = A$, say, conditions \eqref{eqn:k1} yield
\[ \deg(G_1^k - G_2^k) + \deg A \le \deg(R_1 - R_2) \le h. \]
We have
\[ G_1^k - G_2^k = (G_1 - G_2)\sum_{j=0}^{k-1}G_1^jG_2^{k-j-1}. \]
The sum over $j$ is a polynomial of degree $(k-1)d$ with leading coefficient $k$, which does not vanish since $p \nmid k$. As $G_1 \neq G_2$, this implies that 
\[ (k-1)d \leq \deg(G_1^k - G_2^k) \leq h - \deg A < (k+1)d - n, \]
a contradiction. Therefore, $A_1 \ne A_2$.

Thus, \eqref{eqn:k8} is a (univariate) polynomial equation of degree $2k-2$ over $\mathbb F_q[x]$. The number of solutions of such an equation is bounded above by its degree, so once $G_1$, $G_2$ (and hence, $A_1$ and $A_2$) are fixed, there are at most $2k-2$ possibilities for $G_3 \in \mathcal S_{q}(d) \cap \mathcal I$. We conclude that
\[ |\mathcal S_{q}(d) \cap \mathcal I| \le 2k. \]
Therefore, the conclusion of the proposition follows from Lemma \ref{lem2.3} with $\kappa = 2k$ and $\delta = \Delta_k$.
\end{proof}

\subsection*{Acknowledgments} This work is the result of an REU project that took place on the campus of Towson University during the summer of 2022, with the financial support of the National Science Foundation under grants DMS-2136890 and DMS-2149865. The authors also acknowledge financial support from the Fisher College of Science and Mathematics and TU's Mathematics Department. Finally, the authors want to thank the anonymous referee for their thorough reading of the manuscript and for several improvements to the exposition.

\bibliographystyle{amsplain}
\bibliography{refs}

\end{document}